\documentclass{article}

\usepackage[utf8]{inputenc}
\usepackage{lmodern}

\usepackage{enumerate}
\usepackage{amsmath}
\usepackage{amsfonts}
\usepackage{amsthm}
\usepackage{mathtools}
\usepackage{paralist}
\usepackage{xcolor}
\usepackage{hyperref}
\usepackage[affil-it]{authblk}

\usepackage[a4paper, margin=1.2in]{geometry}

\usepackage[maxnames=99]{biblatex}
\title{A Polynomial Ramsey Statement for Bounded VC-dimension}
\author{Tom\'{a}\v{s} Hons}
\affil{\small Computer Science Institute of Charles University, Prague, Czech Republic}
\date{}

\theoremstyle{plain}
\newtheorem{theorem}{Theorem}
\newtheorem{lemma}{Lemma}

\newtheorem{corollary}{Corollary}
\newtheorem{observation}{Observation}

\theoremstyle{definition}
\newtheorem{definition}{Definition}

\theoremstyle{remark}

\newcommand{\N}{\mathbb{N}}

\newcommand{\R}{\mathbb{R}}

\DeclareMathOperator\VCdim{VC-dim}
\newcommand{\eps}{\varepsilon}

\begin{document}

\maketitle

\begin{abstract}
A theorem by Ding, Oporowski, Oxley, and Vertigan states that every sufficiently large bipartite graph without twins contains a matching, co-matching, or half-graph of any given size as an induced subgraph.
We prove that this Ramsey statement has polynomial dependency assuming bounded VC-dimension of the initial graph, using the recent verification of the Erd\H{o}s-Hajnal property for graphs of bounded VC-dimension.
Since the theorem of Ding et al. plays a role in (finite) model theory, which studies even more restricted structures, we also comment on further refinements of the theorem within this context.
\end{abstract}

\section{Introduction}\label{sec:introduction}

Ramsey theory is an area of mathematics asserting that each sufficiently large object contains a still large well-structured subobject.
The foundational theorem, proved by Frank Ramsey in 1930, states that every sufficiently large graph contains a still large homogeneous (induced) subgraph, i.e. a clique or an independent set \cite{Ramsey1930}.
It is well known by a theorem of Erd\H{o}s and Szekeres \cite{Erdos1935} that there is always a logarithmically large homogeneous subgraph.
However, Erd\H{o}s showed by a probabilistic construction that, in general, we cannot hope for larger than logarithmic \cite{Erdos1947}.
Finding the exact base of the logarithm is one of the most difficult and important problems in Ramsey theory, which have recently witnessed a major breakthrough by the work of Campos, Griffiths, Morris, and Sahasrabudhe \cite{Campos2023}, which was subsequently optimized by Gupta, Ndiaye, Norin, and Wei \cite{Gupta2024}.

To obtain a bigger than logarithmically large homogeneous subgraph, it is necessary to put additional assumptions on the given graphs.
A class of graphs $\mathcal{C}$ is said to have the Erd\H{o}s-Hajnal property if there is a constant $e>0$ such that each graph $G \in \mathcal{C}$ on $n$ vertices has a homogeneous subgraph of size $n^e$.
The celebrated Erd\H{o}s-Hajnal conjecture states that for each graph $H$ the class of $H$-free graphs admits the Erd\H{o}s-Hajnal property \cite{Erdos1977}.
The conjecture is widely open in full, and was only recently resolved for all graphs on $5$ vertices \cite{Nguyen2023}\cite{Chudnovsky2013}.

A more restrictive property for a graph class than being $H$-free is having bounded VC-dimension.
The Vapnik–Chervonenkis dimension, VC-dimension for short, is a measure of complexity of combinatorial structures, originally invented in the context of learning theory \cite{Vapnik1971}, which later proved to be fruitful in combinatorics as well.
Structures with bounded VC-dimension often possess better Ramsey properties than general graphs \cite{Fox2021}\cite{Janzer2024}\cite{Balogh2024}.
In particular, Nguyen, Scott, and Seymour recently proved that graphs of bounded VC-dimensions admit the Erd\H{o}s-Hajnal property \cite{NSS_2023}.

In this paper, we consider the following Ramsey statement by Ding, Oporowski, Oxley, and Vertigan \cite{Ding_1996}, which was independently proved by Alekseev \cite{Alekseev1997} (in Russian), and by Gravier, Maffray, Renault, and Trotignon \cite{Gravier2004} building upon a result of Bauslaugh \cite{Bauslaugh2001}.

\begin{theorem}[\cite{Ding_1996}\cite{Alekseev1997}\cite{Gravier2004}]\label{thm:doov}
    There is a monotone unbounded function $f: \N \to \N$ such that the following holds.
    Let $G = (U,V,E)$ be a bipartite graph such that the part $V$ has $n$ vertices and no twins.
    Then it contains an induced subgraph $F=(U',V',E')$ with $|U'| = |V'| \geq f(n)$ that is isomorphic to either:
    \begin{enumerate}[(i)]
        \item a matching,
        \item a co-matching,
        \item a half-graph.
    \end{enumerate}
\end{theorem}

Here the term \emph{twins} refers to a pair of vertices with the same set of neighbors.
All the graphs from the conclusion are bipartite with parts $u_1, \dots, u_n$ and $v_1, \dots, v_n$.
The matching is the graph with $u_i \sim v_j$ iff $i = j$.
The co-matching is the bipartite complement of a matching, i.e. $u_i \sim v_j$ iff $i \not= j$.
The half-graph is the graph with $u_i \sim v_j$ iff $i \leq j$.
Note that the binary complement, i.e. $u_i \sim v_j$ iff $i > j$, of the half-graph of order $n$ contains a half-graph of order $n-1$ as an induced subgraph.

Remarkably, all \cite{Ding_1996}\cite{Alekseev1997}\cite{Gravier2004} proved Theorem~\ref{thm:doov} with log-log dependence of the size of $F$ on the size of $G$, and Gravier et al. showed that we cannot hope for better than logarithmic dependency~\cite{Gravier2004}.
In view of the recent progress on the Erd\H{o}s-Hajnal property, it is a natural question,\footnote{The author learned about the question from Jan Dreier and Szymon Toruńczyk in an open problem session within Algomanet Workshop in September 2024, Warsaw.} whether Theorem~\ref{thm:doov} holds with a polynomial bound under the additional assumption that the graph $G$ has bounded VC-dimension.

Here we confirm that this is indeed true.
Our main result is the following theorem.

\begin{theorem}\label{thm:poly_ding_graphs}
    For every $d \in \N$, there is $c > 0$ such that the following holds.
    Let $G = (U,V,E)$ be a bipartite graph of VC-dimension at most $d$ such that the part $V$ has $n$ vertices and no twins.
    Then it contains an induced subgraph $F=(U',V',E')$ with $|U'| = |V'| \in \Omega(n^c)$ that is isomorphic to either:
    \begin{enumerate}[(i)]
        \item a matching,
        \item a co-matching,
        \item a half-graph.
    \end{enumerate}
\end{theorem}

Moreover, since Theorem~\ref{thm:doov} plays a role in model theory, where it is often considered to be a folklore result \cite{Bonnet2024}\cite{Gajarsky2023}, we comment on further improvements of Theorem~\ref{thm:poly_ding_graphs} within model-theoretic context in Section~\ref{sec:remarks_on_tame_classes}.
In fact, the applications in model theory were the original motivation of this paper.
\section{Preliminaries}\label{sec:prelims}

All our graphs are simple and finite.
All our matrices are binary.
Following the paper of Ding et al. \cite{Ding_1996}, we mostly work in the language of biadjacency matrices.
The biadjacency matrix of a bipartite graph $G = (U,V,E)$ is the matrix $A \in \{0,1\}^{|U| \times |V|}$ with $A_{i,j} = 1$ iff $\{u_i, v_j\} \in E$.
Natural operations for biadjacency matrices include permutation and deletion of rows/columns.
Note that a matrix $B$ that is obtained from $A$ by these operations is the biadjacency matrix of an induced subgraph of $G$.
We call such a matrix $B$ a \emph{submatrix} of $A$.

Given values $\alpha, \beta, \gamma \in \{0,1\}$, we say that a matrix $A \in \{0,1\}^{m \times n}$ is a \emph{$(\alpha, \beta, \gamma)$-matrix} if it is of the form:
\begin{align*}
    A_{i,j} =
    \begin{cases}
        \alpha &\text{ if } i < j, \\
        \beta  &\text{ if } i = j, \\
        \gamma &\text{ if } i > j.
    \end{cases}
\end{align*}

If not all the parameters $\alpha, \beta, \gamma$ are equal, we say that the $(\alpha, \beta, \gamma)$-matrix is \emph{inhomogeneous}.
Observe that the inhomogeneous square $(\alpha, \beta, \gamma)$-matrices correspond to the biadjacency matrices of the graphs from the conclusion of Theorem~\ref{thm:poly_ding_graphs} (and the binary complement of a half-graph).
Consequently, our result can be stated in the matrix language as follows.

\begin{theorem}\label{thm:poly_ding_matrices}
    For every $d \in \N$, there is $c > 0$ such that the following holds.
    Let $A$ be a binary matrix of VC-dimension at most $d$ with at least $n$ columns, no two of which are identical.
    Then it contains an inhomogeneous square $(\alpha, \beta, \gamma)$-submatrix of size $n^c \times n^c$.
\end{theorem}

In particular, we ensure that $\alpha \not= \beta$.
The VC-dimension of a matrix is defined in the next section.

\subsection{VC-dimension}

The Vapnik–Chervonenkis dimension is a numeric parameter capturing the complexity of a set system \cite{Vapnik1971}.

\begin{definition}[VC-dimension of a set system]
    Consider a set system $\mathcal{S} = (U, \mathcal{F})$, where $U$ is a finite set and $\mathcal{F}$ is a family of subsets of $U$.
    We say that $X \subseteq U$ is \emph{shattered} if
    \[
        \big| \{X \cap F: F \in \mathcal{F} \} \big| = 2^{|X|}
        .
    \]
    The \emph{VC-dimension of the set system} $\mathcal{S}$, denoted by $\VCdim(\mathcal{S})$, is the cardinality of the largest shattered subset of $U$.
\end{definition}

There are several natural set systems, whose VC-dimension is of our interest.
For a graph $G$, we have the system $(V, \{N(v) : v \in V\})$ of neighborhoods of $G$.
More generally, for a matrix $A \in \{0,1\}^{m \times n}$, we consider the system of columns $([m], \{A_{*,j} : j \in [n]\})$, where the column $A_{*,j}$ stands for the set $\{i \in [m] : A_{i,j} = 1\}$.
Similarly, the matrix $A$ has an associated system of rows $([n], \{A_{i,*} : i \in [m]\})$.
Note that if $A$ is the adjacency matrix of a graph $G$, then the column and row systems of $A$ are the same (as $A$ is symmetric), and they agree with the system of neighborhoods of $G$.

We use these systems to define the VC-dimension of graphs and matrices.

\begin{definition}[VC-dimension of a graph]
    The \emph{VC-dimension of a graph} $G = (V,E)$, $\VCdim(G)$, is defined as the VC-dimension of the system $(V, \{N(v) : v \in V\})$ of neighborhoods of $G$.
\end{definition}

\begin{definition}[VC-dimension of a matrix]
    The \emph{VC-dimension of a matrix} $A \in \{0,1\}^{m \times n}$, $\VCdim(A)$, is defined as the maximum of the VC-dimensions the system of its columns $([m], \{A_{*,j} : j \in [n]\})$ and the system of its rows $([n], \{A_{i,*} : i \in [m]\})$.
\end{definition}

As hinted, the definition of the VC-dimension is consistent for a graph and its adjacency matrix.
Furthermore, the definitions also agree for a bipartite graph and its biadjacency matrix.

\begin{observation}
    Let $G$ be a graph and $A$ its adjacency matrix, then $\VCdim(G) = \VCdim(A)$.
    Moreover, if $G$ is bipartite and $B$ is its biadjacency matrix, then $\VCdim(G) = \VCdim(B)$.
\end{observation}

VC-dimension is monotone with respect to taking submatrices (i.e. induced subgraphs) and their complements.

\begin{observation}\label{obs:monotonicity_of_vc_dim}
    If $B$ is (the binary complement of) a submatrix of $A$ (or its transpose), then $\VCdim(B) \leq \VCdim(A)$.
\end{observation}

\subsection{Growth function}

A closely related notion to VC-density is the \emph{growth function} of a system.

\begin{definition}[Growth function of a set system]
    Let $\mathcal{S} = (U, \mathcal{F}), |U| = N,$ be a set system.
    The growth function $\pi_\mathcal{S}: [N] \to \N$ is defined as
    \[
        \pi_\mathcal{S}(n) = 
        \max_{
            X \subseteq U, |X| = n
        } 
        \big| \{X \cap F: F \in \mathcal{F} \} \big|
        .
    \]
\end{definition}

The classical Sauer-Shelah lemma states that systems of bounded VC-dimension have the growth function bounded by a polynomial \cite{Sauer_1972}\cite{Shelah_1972}.

\begin{theorem}[Sauer-Shelah lemma \cite{Sauer_1972}\cite{Shelah_1972}]\label{thm:sauer_shelah}
    Let $\mathcal{S} = (U, \mathcal{F})$ be a set system of VC-dimension at most $d$.
    Then for any set $X \subseteq U$ of size $n$, we have
    \[
        \big| \{X \cap F: F \in \mathcal{F} \} \big| \leq \Phi_d(n) \coloneqq \sum_{i=0}^d \binom{n}{i}
        .
    \]
\end{theorem}

\subsection{Erd\H{o}s-Hajnal property}

As stated in the introduction, we say that a class of graphs $\mathcal{C}$ has the Erd\H{o}s-Hajnal property if the graphs from $\mathcal{C}$ admit polynomially large homogeneous subgraph.
A recent theorem of Nguyen, Scott, and Seymour that verifies the Erd\H{o}s-Hajnal property for graphs of bounded VC-dimension \cite{NSS_2023} plays a crucial part in our proof.

\begin{theorem}[\cite{NSS_2023}]\label{thm:poly_ramsey}
    For every $d \geq 1$, there exists $e(d) > 0$ such that every graph $G$ on $n$ vertices of VC-dimension at most $d$ contains a homogeneous subgraph of size at least $n^{e(d)}$.
\end{theorem}    
\section{Proof}\label{sec:proof}

Given a matrix of VC-dimension at most $d$ with at least $n$ columns, no two of which are identical, our goal is to find an inhomogeneous square $(\alpha, \beta, \gamma)$-submatrix with $\alpha \not= \beta$ of polynomial size.
The original proof of Ding et al. is carried out in two steps.
First, they obtain a square $(\alpha, \beta, *)$-submatrix of the given matrix with $\alpha \not= \beta$, where $*$ is a Joker (i.e. no restriction for entries below the diagonal).
As the second step, they apply the Ramsey theorem to an auxiliary graph to homogenize the lower triangular matrix.
Both steps shrink the matrix logarithmically.

In our case, we prove that the auxiliary graph constructed from a matrix of bounded VC-dimension has bounded VC-dimension as well (Lemma~\ref{lem:vc_dim_of_upper_graph}).
This allows us to perform the second step from the $(\alpha, \beta, *)$-submatrix to the $(\alpha, \beta, \gamma)$-submatrix with a polynomial dependence by Theorem~\ref{thm:poly_ramsey}.

Moreover, we show that after a convenient preparatory step, it is possible to use the Ramsey theorem, i.e. in our case Theorem~\ref{thm:poly_ramsey}, in the first step to obtain the $(\alpha, \beta, *)$-submatrix.
Indeed, we initially produce a \emph{switch submatrix} (Lemma~\ref{lem:switch_matrix}), which allows us to choose the value $\beta$ on the diagonal \emph{after} we obtain the value $\alpha$ above the diagonal from the Ramsey theorem (Lemma~\ref{lem:ab*_matrix}).
Then it remains to homogenize the values below the diagonal as mentioned above (Lemma~\ref{lem:abc_matrix}).

\subsection{Switch submatrix}

We begin by defining the notion of a \emph{switch matrix}.
Then we show that under the assumptions of Theorem~\ref{thm:poly_ding_matrices}, the given matrix contains a polynomially large switch submatrix.

\begin{definition}
    A matrix $A \in \{0,1\}^{n \times 2n}$ is a \emph{switch matrix of size $n$} if it is of the following form:
    \begin{enumerate}[(i)]
        \item for all $i \in [n]$, we have $A_{i,2i-1} = 0$ and $A_{i,2i} = 1$,
        \item for all $i,j \in [n], i < j$, we have $A_{i,2j-1} = A_{i,2j}$.
    \end{enumerate}
\end{definition}

The following important technical lemma says that in our setting of bounded VC-dimensions we have polynomially large switch submatrices.
We present the lemma for a general bounding function $f: \R^+_0 \to \R^+_0$ that satisfies some reasonable assumptions. 
That is, we assume that $f$ is convex with $f(0) = 0$ and $f(x) \geq 2x$ for all $x \in \R^+_0$.
Note that this implies $f$ to be strictly increasing (by Jensen's inequality).
Thus, $f$ has a well-defined (concave) inverse $f^{-1}$.

\begin{lemma}\label{lem:switch_matrix}
    Let $A$ be a matrix whose system of columns has a growth function $\pi \leq f$ with $f$ as above.
    Suppose $A$ has $n$ columns, no two of which are identical.
    Then $A$ contains a switch submatrix $B$ of size at least $\frac{1}{2} f^{-1}(n)$.
\end{lemma}
\begin{proof}
    Starting from the empty submatrix $B_0$, we greedily construct switch submatrices $B_1, B_2, \dots$, with $B_i$ having size $i$.
    Each $B_{i+1}$ is obtained by extending $B_i$ by a new row and two columns.
    We only need to prove that we can run the process for at least $\frac{1}{2} f^{-1}(n)$ steps.
    
    Suppose we have a switch submatrix $B_i$ of $A$ of size $i$ that uses rows $\mathcal{R}$ and columns $\mathcal{C}$ from the original matrix (likely in different order).
    For $x \in \{0,1\}^{\mathcal{R}}$, let $\Gamma(x)$ be the set of columns of $A$ that have the pattern $x$ at rows $\mathcal{R}$.
    To extend the submatrix $B_i$, we want to find a vector $x^* \in \{0,1\}^{\mathcal{R}}$ such that $|\Gamma(x^*) \setminus \mathcal{C}| \geq 2$.
    This is indeed sufficient: take distinct columns $C, C' \in \Gamma(x^*) \setminus \mathcal{C}$.
    By assumption, these columns differ at some row $R$, which certainly lies outside of $\mathcal{R}$.
    Thus, we can extend $B_i$ by the row $R$ and columns $C, C'$ to form the submatrix $B_{i+1}$.
    To be precise, the last column is the one of $C$ or $C'$ that has the value $1$ in the row $R$.
    Clearly, the resulting submatrix $B_{i+1}$ is then a switch matrix of size $i+1$.
    
    It remains to show that such a vector $x^* \in \{0,1\}^{\mathcal{R}}$ exists when extending the submatrix $B_i$ with $i < \frac{1}{2} f^{-1}(n)$.
    We have at most $\pi(i)$ vectors $x \in \{0,1\}^{\mathcal{R}}$ for which the set $\Gamma(x)$ is non-empty.
    Therefore, the average size of a non-empty $\Gamma(x) \setminus \mathcal{C}$ is at least
    \[
        \frac{n-|\mathcal{C}|}{\pi(i)} \geq \frac{n - 2i}{f(i)} = \frac{n}{f(i)} - \frac{2i}{f(i)}
        ,
    \]
    using $\pi(i) \leq f(i)$ and $|\mathcal{C}| = 2i$.
    Clearly, $\frac{2i}{f(i)} \leq 1$ due to the assumption $f(x) \geq 2x$.
    We proceed to lower-bound the other fraction.
    We have
    \[
        \frac{n}{f(i)} > \frac{n}{f(\frac{1}{2} f^{-1}(n))} = 2 \cdot\frac{\frac{1}{2}n}{f(\frac{1}{2} f^{-1}(n))}
        ,
    \]
    where the inequality follows from $i < \frac{1}{2} f^{-1}(n)$ and the fact that $f$ is strictly increasing.
    We claim that the last fraction is at least $1$.
    Indeed, this statement is equivalent to
    \[
        f\left(\frac{1}{2} f^{-1}(n)\right) \leq \frac{1}{2}n 
        \quad \Leftrightarrow \quad
        \frac{1}{2} f^{-1}(n) \leq f^{-1}\left(\frac{1}{2}n\right)
        \quad \Leftrightarrow \quad        
        \frac{f^{-1}(0) + f^{-1}(n)}{2} \leq f^{-1}\left(\frac{0 + n}{2}\right)
        ,
    \]
    using that $f^{-1}(0) = 0$, where the last inequality is exactly the Jensen inequality applied to the concave function $f^{-1}$.
    To conclude, we obtain
    \[
        \frac{n-|\mathcal{C}|}{\pi(i)} \geq \frac{n}{f(i)} - \frac{2i}{f(i)} > 2 \cdot 1 - 1 = 1
        ,
    \]
    and therefore, $|\Gamma(x^*) \setminus \mathcal{C}| \geq 2$.
\end{proof}

We get an immediate corollary for matrices of bounded VC-dimension.

\begin{corollary}\label{cor:switch_matrix_for_bounded_VC}
    Let $A$ be a matrix of VC-dimension at most $d$ with $n$ columns, no two of which are identical.
    Then $A$ contains a switch submatrix $B$ of size $\Omega(n^{1/d})$.
\end{corollary}
\begin{proof}
    By Theorem~\ref{thm:sauer_shelah}, the class of matrices of VC-dimension at most $d$ has the growth function bounded by $\Phi_d(n) \in \mathcal{O}(n^d)$.
    Clearly, a function $cn^d \geq \Phi_d(n)$ with $c \geq 2, d \geq 1$ satisfies all the assumptions on the function $f$ from Lemma~\ref{lem:switch_matrix}.
    Hence, we obtain a switch submatrix of $B$ of the size $\frac{1}{2c} n^{1/d}$ by Lemma~\ref{lem:switch_matrix}.
\end{proof}

\subsection{Upper graph and its VC-dimension}

Here we define the notion of an auxiliary graph called \emph{upper graph} associated to a square matrix $A$.
We prove that if $A$ has bounded VC-dimension, the corresponding upper graph has bounded VC-dimension as well.

\begin{definition}
    Let $A$ be a square matrix of size $n$.
    We call the graph $G = ([n], \{\{i,j\} : i < j, A_{i,j} = 1\})$ the \emph{upper graph} of the matrix $A$.
\end{definition}

In other words, the upper graph of the matrix $A$ is the graph $G$ whose adjacency matrix agrees with $A$ in the entries above the diagonal.

\begin{lemma}\label{lem:vc_dim_of_upper_graph}
    Let $A \in \{0,1\}^{n \times n}$ be a square $(*, 0, *)$-matrix and $G$ be the upper graph of $A$.
    For any $d \in \N$, we have that $\VCdim(G) \geq 4d$ implies $\VCdim(A) \geq d$.
    
    Therefore, if $\VCdim(A) \leq d$, then $\VCdim(G) \leq 4d+3$.
\end{lemma}
\begin{proof}
    Let $B$ be the adjacency matrix of $G$.
    The VC-dimension of $G$ is at least $4d$, which is witnessed by a shattered set $X = \{v_1 < \dots < v_{4d}\} \subseteq [n]$ of size $4d$.
    Let $Y = \{u_1 < \dots < u_{2^{4d}}\} \subseteq [n]$ be the set of vertices whose neighborhoods witness that $X$ is shattered.

    We divide the set $X$ into $X_1 = \{v_1 < \dots < v_{2d}\}$ and $X_2 = \{v_{2d+1} < \dots < v_{4d}\}$.
    Moreover, we split the set $Y$ into $Y_1 = \{u_i \in Y : u_i < v_{2d}\}$ and $Y_2 = \{u_i \in Y : u_i \geq v_{2d}\}$.
    Certainly, one of the sets $Y_1$ or $Y_2$ has size at least $2^{4d-1}$.
    If $|Y_2| \geq 2^{4d-1}$, the submatrix $C$ of $B$ induced by the rows $X_1$ and columns $Y_2$ lies entirely in the upper triangle of $B$, so it is a submatrix of $A$ as well.
    In the other case, if $|Y_1| \geq 2^{4d-1}$, the submatrix $D$ of $B$ induced by $X_2 \times Y_1$ lies in the bottom triangle of $B$.
    By the symmetry of $B$, the transpose of $D$ appears in the upper triangle of $B$ and hence also in $A$.
    The rest of the proof is then the same with the submatrix $D^\top$ instead of $C$ and with the role of rows and columns reversed.
    
    Suppose that $|Y_2| \geq 2^{4d-1}$.
    As we argued in the previous paragraph, the submatrix $C$ of $B$ induced by the rows $X_1$ and columns $Y_2$ is a submatrix of $A$ as well.
    We claim that the rows $X_1$ of $A$, in particular the submatrix $C$, witness that $\VCdim(A) \geq d$.
    The matrix $C$ contains at least $2^{4d - 1} / \, 2^{2d} = 2^{2d-1}$ distinct columns since it has at least $2^{4d-1}$ columns in total and any pattern may occur in at most $2^{|X_2|} = 2^{2d}$ of them.
    From Theorem~\ref{thm:sauer_shelah}, we obtain the inequality
    \[
        2^{2d-1} \leq \text{ \#distinct patterns on $X_1$ } \leq \Phi_{\VCdim(A)}(2d) = \sum_{i=0}^{\VCdim(A)} \binom{2d}{i}
        .
    \]
    This implies $\VCdim(A) \geq d$ as we have the equality $\Phi_{d}(2d) = 2^{2d-1}$ by basic properties of binomial coefficients.
    
    The part ``Therefore'' is the contrapositive that includes the values of $\VCdim(G)$ that are not divisible by $4$.
\end{proof}

\subsection{Homogenization}

In this part, we utilize Theorem~\ref{thm:poly_ramsey} to first extract an $(\alpha, \beta, *)$-matrix with $\alpha \not= \beta$ from a switch matrix and then an $(\alpha, \beta, \gamma)$-matrix from the $(\alpha, \beta, *)$-matrix.
The polynomial dependency in Theorem~\ref{thm:poly_ramsey} ensures polynomial dependency in our statements.

\begin{lemma}\label{lem:ab*_matrix}
    Let $A$ be a switch matrix of size $n$ of VC-dimension at most $d$.
    Then it contains a square $(\alpha, \beta, *)$-matrix with $\alpha \not= \beta$ of size at least $n^{e(4d+3)}$, where $e(\cdot)$ is the function from Theorem~\ref{thm:poly_ramsey}.
\end{lemma}
\begin{proof}
    First, we consider the square $(*,0,*)$-submatrix $B$ of $A$ obtained by deleting all even columns of $A$.
    Let $G$ be the upper graph of $B$.
    By Observation~\ref{obs:monotonicity_of_vc_dim} and Lemma~\ref{lem:vc_dim_of_upper_graph}, the VC-dimension of $G$ is at most $4d+3$.
    Theorem~\ref{thm:poly_ramsey} then implies that $G$ contains a homogeneous subgraph on a vertex set $X$ of size at least $n^{e(4d+3)}$.
    We consider the submatrix $C$ of $A$ with rows $\mathcal{R} = X$, and columns $\mathcal{C} = \{2i - 1 : i \in X\}$ if $G[X]$ is a clique, or $\mathcal{C} = \{2i : i \in X\}$ if $G[X]$ is an independent set.
    Clearly, $C$ is a square matrix of size at least $n^{e(4d+3)}$.
    
    We claim that the matrix $C$ is an $(\alpha, \beta, *)$-matrix with $\alpha \not= \beta$.
    If $G[X]$ is a clique, then for each $i,j \in X, i < j$, we have that $A_{i,2j-1} = 1$.
    Moreover, by definition of a switch matrix, we have for all $i \in X$ that $A_{i,2i-1} = 0$.
    Therefore, the matrix $C$ is a $(1,0,*)$-matrix.
    
    If $G[X]$ is an independent set, then for each $i,j \in X, i < j$, we have that $A_{i,2j-1} = 0$.
    Thus, also $A_{i,2j} = 0$ by definition of a switch matrix.
    Moreover, we have for all $i \in X$ that $A_{i,2i} = 1$.
    Therefore, the matrix $C$ is a $(0,1,*)$-matrix.
\end{proof}

\begin{lemma}\label{lem:abc_matrix}
    Let $A$ be a square $(\alpha, \beta, *)$-matrix of size $n$ of VC-dimension at most $d$.
    Then it contains a square $(\alpha, \beta, \gamma)$-matrix of size at least $n^{e(4d+3)}$, where $e(\cdot)$ is the function from Theorem~\ref{thm:poly_ramsey}.
\end{lemma}
\begin{proof}
    We proceed similarly as in the proof of Lemma~\ref{lem:ab*_matrix}, except that we construct the upper graph from the transpose of the matrix $A$, or from the binary complement of the transpose if $A$ is a $(0,1,*)$-matrix (to formally comply with Lemma~\ref{lem:vc_dim_of_upper_graph} by having $0$ on the diagonal).
    By applying Theorem~\ref{thm:poly_ramsey}, we obtain a square $(\alpha, \beta, \gamma)$-submatrix of $A$ of size at least $n^{e(4d+3)}$.
\end{proof}

It remains to prove the main theorem.

\begin{proof}[Proof of Theorem~\ref{thm:poly_ding_matrices}]
    By composing Corollary~\ref{cor:switch_matrix_for_bounded_VC} with Lemmas~\ref{lem:ab*_matrix}~and~\ref{lem:abc_matrix}, we obtain an inhomogeneous square $(\alpha, \beta, \gamma)$-submatrix with $\alpha \not= \beta$ of polynomial size with the exponent
    \[
        \frac{1}{d} \cdot e(4d+3) \cdot e(4d+3)
        ,
    \]
    where $e(\cdot)$ is the function from Theorem~\ref{thm:poly_ramsey}.
\end{proof}

\section{Remarks on tame classes}\label{sec:remarks_on_tame_classes}

In recent years, much effort has been invested into combinatorial understanding of classical model-theoretic notions such as stability and dependence with the aim to establish dividing lines between tame and wild classes of finite structures \cite{Gajarsky2023}\cite{Torunczyk2023}\cite{Dreier2024}.
These include the model-theoretic notions of (monadic and edge-)stability, (monadic) dependence and the combinatorial notion of bounded twin-width.
For their definitions please see \cite{Tent2012}\cite{Simon2015}\cite{Dreier2024}\cite{Bonnet2021}, Figure~\ref{fig:tameness_inclusions} shows their inclusions.
Considered folklore, Theorem~\ref{thm:doov} plays an important role in this area.
For example, it was one of the tools that allowed Dreier, M\"{a}hlmann, Siebertz, and Toruńczyk to give the first purely combinatorial characterization of monadic stability \cite{Dreier2023c}.
Moreover, Dreier, M\"{a}hlmann, and Toruńczyk were able to describe the forbidden subgraphs of monadically dependent graph classes, yielding the first combinatorial characterization of monadic dependence \cite{Dreier2024}.

\begin{figure}[t]
    \centering
    \includegraphics[scale=1]{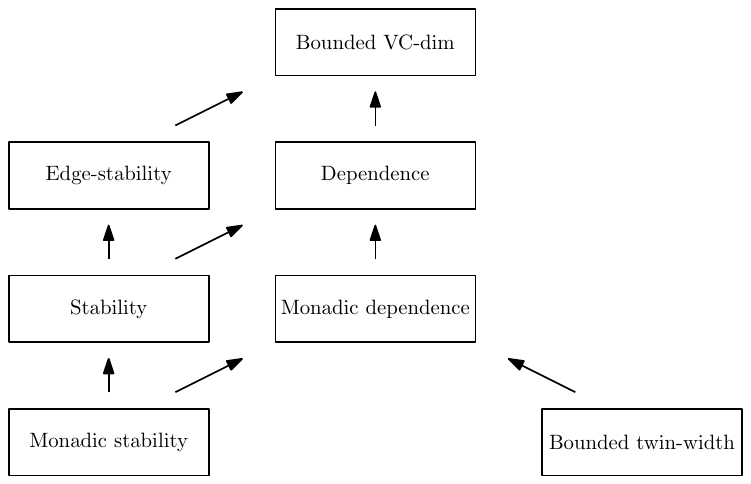}
    \caption{A diagram of various tameness notions and their inclusions.}
    \label{fig:tameness_inclusions}
\end{figure}

All the notions in Figure~\ref{fig:tameness_inclusions} imply bounded VC-dimension.
Therefore, our improvement in form of Theorem~\ref{thm:poly_ding_graphs} applies to all these classes.
However, the more restricted notions allow for further improvements of Theorem~\ref{thm:poly_ding_graphs} and better bounds.

In Corollary~\ref{cor:switch_matrix_for_bounded_VC}, we used Theorem~\ref{thm:sauer_shelah} to bound the growth function of classes of bounded VC-dimension.
While Theorem~\ref{thm:sauer_shelah} gives a polynomial bound on the growth function of the neighborhood systems, also called neighborhood complexity, it is conjectured that the monadically dependent classes admit an \emph{almost linear} bound \cite{Dreier2023a}\cite{Dreier2024}.
That is, that for graphs from the class $\mathcal{C}$, the growth function is bounded by $cn^{1+\eps}$ for every $\eps > 0$ a constant $c = c(\eps, \mathcal{C})$.
This is known to be true for monadically stable classes \cite{Dreier2023b}.
For classes of bounded twin-width, the situation is even more pleasant as the neighborhood complexity is linear as proved independently by \cite{Bonnet2022} and \cite{Przybyszewski2023}, with an essentially tight bound by \cite{Bonnet2024a}.
Using these exponents in Lemma~\ref{lem:switch_matrix}, we conclude that graphs from classes of bounded twin-width admit a switch submatrix of a linear size, while monadically stable classes admit a polynomial switch submatrix with the exponent $1 - \eps$ for every $\eps > 0$.

Moreover, if the upper graph of the given matrix belongs to a tame class, we can reach better bounds by substituting Theorem~\ref{thm:poly_ramsey} by a more specific statement.
This is the case of classes of bounded twin-width,\footnote{Similarly as in Lemma~\ref{lem:vc_dim_of_upper_graph}, large twin-width of $G$ is witnessed by a certain submatrix called mixed-minor \cite{Bonnet2021}. Then $A$ contains a mixed-minor of at least half of the size, and consequently has large twin-width as well. The contrapositive states that the upper graph of a matrix with bounded twin-width has itself bounded twin-width. However, a more direct argument with a better quantitative bound would be desirable.} which admit even the strong Erd\H{o}s-Hajnal property \cite[Theorem 20]{Bonnet2024b}.
It was implicitly proved by Malliaris and Shelah \cite[Theorem~3.5]{Malliaris2013} that stable graph classes possess the Erd\H{o}s-Hajnal property.
The bound was further improved in \cite[Corollary~4.9]{Malliaris2017} using only the milder assumption of edge-stability (that is, only the edge relation is required to be stable; not necessarily all the definable relations).
Furthermore, for the more restricted monadically stable graph classes, it was recently proved that they possess the Erd\H{o}s-Hajnal property with the exponent $1/2 - \eps$ for any $\eps > 0$ \cite[Corollary~62]{Braunfeld2025}.
Notably, this says that there is a uniform Erd\H{o}s-Hajnal exponent for all monadically stable graph classes.

Speaking about (edge-)stable classes, we should point out that we can further simplify the conclusion of Theorem~\ref{thm:poly_ding_graphs}.
Indeed, for each edge-stable graph class there is (by definition) a bound on the size of half-graphs that can appear as subgraphs, which allows us to narrow the conclusion of Theorem~\ref{thm:poly_ding_graphs} that the target subgraph is either a matching or a co-matching.

\section{Acknowledgement}\label{sec:acknowledgement}

The author would like to thank Patrice Ossona de Mendez for his many suggestions concerning Section~\ref{sec:remarks_on_tame_classes}.
Many thanks also belong to the reviewers whose helpful comments have
helped improve the presentation.

The author is partially supported by the European Research Council (ERC) under the European Union’s Horizon 2020 research and innovation programme (ERC Synergy Grant DYNASNET, grant agreement No 810115), by Project 24-12591M of the Czech Science Foundation (GAČR), and by the grant SVV–2025–260822.

\printbibliography

\end{document}